\documentclass[a4paper,11pt,reqno] {amsart}

\usepackage[utf8]{inputenc}
\usepackage[english]{babel}
\usepackage{hyperref}
\usepackage{amsmath}
\usepackage{amsthm}
\usepackage{amsfonts}
\usepackage{amssymb}
\usepackage{enumerate}
\usepackage{enumitem}
\usepackage{xcolor}
\usepackage{graphicx}
\usepackage{mathrsfs}

\numberwithin{equation}{section}
\allowdisplaybreaks

\reversemarginpar
\theoremstyle{plain}
\newtheorem{theorem}{Theorem}[section]
\newtheorem*{theoremA}{Theorem A}
\newtheorem*{theoremB}{Theorem B}
\newtheorem*{question1}{Question 1}
\newtheorem*{question2}{Question 2} \newtheorem*{question3}{Question 3}
\newtheorem*{question4}{Question 4}
\newtheorem*{question5}{Question 5}

\newtheorem{proposition}[theorem]{Proposition}
\newtheorem{lemma}[theorem]{Lemma}
\newtheorem{corollary}[theorem]{Corollary}
\theoremstyle{definition}
\newtheorem{definition}[theorem]{Definition}
\newtheorem{remark}[theorem]{Remark}
\newtheorem{example}[theorem]{Example}

\newcommand{\N}{\mathbb{N}}
\newcommand{\C}{\mathbb{C}}

\newcommand{\pa}[1]{\left( #1\right)}
\newcommand{\co}[1]{\left[ #1\right]}

\newcommand{\frF}{\mathfrak{F}}
\newcommand{\frR}{\mathfrak{R}}
\newcommand{\p}{{p_*}} 
\newcommand{\q}{{q_*}}  
\newcommand{\s}{\sigma}
\newcommand{\oneq}{\frac 1\q}

\newcommand{\abs}[1]{\lvert #1\rvert}

\newcommand{\sm}{\setminus}

\newcommand{\tn}[1]{\textnormal{#1}}

\newcommand{\dist}{\operatorname{dist}}

\newcommand{\CC}{\mathcal{C}} 
\newcommand{\DD}{\mathcal{D}}

\newcommand{\eps}{\epsilon}

\newcommand{\Om}{\Omega}
\newcommand{\al}{\alpha}
\newcommand{\be}{\beta}
\newcommand{\de}{\delta}
\newcommand{\la}{\lambda}
\newcommand{\si}{\sigma}

 \newcommand{\R}{\mathbb{R}}

\renewcommand{\epsilon}{\varepsilon}
\renewcommand{\phi}{\varphi}

\renewcommand{\ss}{\subset}
\renewcommand{\kappa}{\varkappa} 

\newcommand{\Li}{L}

\begin{document}

\title[Estimates of growth of subharmonic and analytic functions] 
{Self-improving estimates of growth \\ of 
subharmonic and analytic functions}

\author[G. Bello]{Glenier Bello}
\address{Departamento de Matem\'{a}ticas e
Instituto Universitario de Matem\'{a}ticas y Aplicaciones\\
Universidad de Zaragoza\\
50009 Zaragoza\\
Spain}
\email{gbello@unizar.es}
\author[D. V. Yakubovich]{Dmitry Yakubovich}
\address{
Departamento de Matem\'aticas\\
Universidad Aut\'onoma de Madrid\\
Cantoblanco, 28049 Madrid, Spain}
\email {dmitry.yakubovich@uam.es}
\date{\today}

\subjclass[2020]{30A10, 31B05}
\keywords{Estimates of subharmonic functions, estimates of analytic functions, 	
boundary growth}

\begin{abstract}
Given a bounded open subset $\Om$ and closed subsets $A,B$ 
of $\R^k$, we discuss when an estimate 
$u(x)\le g(\dist(x,A\cup B))$, $x\in\Omega\sm(A\cup B)$, 
for a function $u$ subharmonic on $\Om\sm B$, 
implies that 
$u(x)\le h(\dist(x,B))$,  $x\in\Omega\sm B$, 
where $g,h:(0,\infty)\to (0,\infty)$ are decreasing functions and $g(0^+)=h(0^+)=\infty$. We seek for explicit expressions of $h$ in terms of $g$. We give some results of this type and show that Domar's work \cite{Domar} permits one to deduce other results in this direction. Then we compare these two approaches. Similar results are deduced for estimates of analytic functions. 
\end{abstract}
\thanks{Both authors acknowledge the support of the Spanish Ministry for Science and Innovation under Grant PID2022-137294NB-I00, DGI-FEDER.  G. Bello has also been partially supported by  PID2022-138342NB-I00 for \emph{Functional Analysis Techniques in Approximation Theory and Applications (TAFPAA)} and by Project E48\_23R, D.G. Arag\'{o}n}
\maketitle

\section{Introduction}
The main question addressed in this paper is as follows. 
Assume  $\Om\subset \R^k$ is a bounded open set, $A\subset \Om$ is compact and $B\subset \R^k$ is closed. 
Let $u:\Om\sm B\to[-\infty,\infty)$ be a subharmonic function that satisfies an estimate 
\begin{equation}\label{eq:estim-u1}
	u(x)\le g(\dist(x,A\cup B)), \quad x\in\Omega\sm(A\cup B), 
\end{equation}
where $g:(0,+\infty)\to (0,+\infty)$ is a 
strictly decreasing function with $g(0^+)=+\infty$. 
We wish to deduce an estimate 
\begin{equation}\label{eq:estim-u-dist-B}
u(x)\le h(\dist(x,B)),  
\end{equation}
where the function $h:(0,+\infty)\to (0,+\infty)$ is defined explicitly in terms of $g$. The first part of the paper (Sections~\ref{Sec:our-main-results} and \ref{Sec:proof-2.4}) is devoted to this problem.  

Theorem~\ref{thm:mainP1}, together with Remark~\ref{rem:thm-2.4-regular-g}, show that whenever $A$ is a Lipschitz curve and $g$ is regular and 
its singularity at $0^+$ is of power type, \eqref{eq:estim-u1} implies that for some constant $C>0$,  
\begin{equation}\label{eq:est-gC-dist}
u(x)\le g(C\dist(x,B)), \quad x\in \Om\setminus B.  
\end{equation} 
If $g(s)$ is regular and its growth at $0^+$ 
is slower than any power $s^{-\eps}$, $\eps>0$, then 
the estimate~\eqref{eq:est-gC-dist} is stronger than an estimate  $u(x)\le Cg(\dist(x,B))$. 

The above-mentioned results, applied to domains $\Om$ in the complex plane, allow us to get results on estimates for analytic functions. In particular, Theorem~\ref{thm:mainP1} implies that 
for an analytic function $f:\Om\to\C$, the estimate  $|f(z)|\le C(\dist(z,A\cap B))^{-\beta}$, where $A\subset \Om$ is a compact Lipschitz curve implies that 
$|f(z)|\le C_1(\dist(z,B))^{-\beta}$, for some constant $C_1$ and the same exponent $\beta>0$. See Corollary~\ref{cor-analytic-f} below.

The second part of the paper (which starts with 
Section~\ref{Sec:quant-vers-Domar}) 
is devoted to the discussion of the results by Domar~\cite{Domar} and obtaining their quantitative versions.  We show that these later results give an alternative approach to the above problem.

Domar's setting is as follows. 
Given a nonnegative upper semicontinuous function $F\colon\Omega\to(-\infty,+\infty]$, where $\Om\ss\R^k$, 
denote by $\CC_F$ the class of all subharmonic functions $u\colon\Omega\to\R$ 
such that $u\le F$ and set 
\begin{equation}\label{eq:M(x)}
M(x):=\sup\{u(x):u\in\CC_F\}. 
\end{equation}
Domar addresses the question whether $M(x)$ is bounded on compact subsets of $\Om$. 
It is well-known that it is true if and only if $M$ is subharmonic. 
For $k=2$, this question is related to normality of the family of analytic functions $f$ in a domain $\Om$ in the complex plane, defined by the  inequality $\log |f|\le F$.

Domar gives two theorems in this direction. 
Under the same assumptions as Domar, in Section~\ref{Sec:quant-vers-Domar} we give  estimates of the form 
\begin{equation}\label{eq:estim-M}
M(x)\le h(\dist(x,\partial\Om)),  
\end{equation}
where $h$ is a function explicitly written down in terms of $F$. This is done by following Domar's arguments.

In Section~\ref{Sec:comparison}, we formulate and prove Theorem~\ref{thm-Domar-A-B}, which gives an answer to the main question of the paper under the assumption that $A$ is $\p$-admissible for some $\p<k$ 
(see Definition~\ref{def-p-adm} below).  Roughly speaking, this means that $A$ is not too massive, in particular, its Minkowski dimension has to be less than $k$. If the Assouad dimension of $A$ is less than $\p<k$, then $A$ is $\p$-admissible. 

Theorem~\ref{thm-Domar-A-B} is obtained by applying our quantitative versions of Domar's results. 

In several places, we find out that a regularity  assumption on the growth of $g$ permits us to formulate simpler results. This assumption is that $g(1/s)$ belongs to a Hardy field (see Section~\ref{Sec:our-main-results} for a definition). 

Theorem~\ref{thm-Domar-A-B} applies to functions $g$ that may have much stronger singularity at $0^+$ than the singularity of power type. 
In Section~\ref{Sec:comparison}, we will prove 
Theorem~\ref{thm-Domar-s^-b}, which is a particular case of 
Theorem~\ref{thm-Domar-A-B} for regular functions $g$ 
whose singularity at $0^+$ is at most of power type. 
For these functions $g$, it asserts that whenever $A$ is $\p$-admissible and $\p<k$,~\eqref{eq:estim-u1} implies that 
for any $a>1$, there exists $C\in (0,1)$ such that 
\begin{equation}
	\label{estim-cg0}
	u(x)\le a g(C\dist(x,B)), \quad x\in \Om\sm B.  
\end{equation}

It is interesting to compare Theorem~\ref{thm:mainP1}, obtained by a direct estimate, with Theorem~\ref{thm-Domar-A-B}, which follows from Domar's methods. On one hand, Theorem~\ref{thm-Domar-A-B}
allows much higher growth of $g$ at $0$ than Theorem~\ref{thm:mainP1} and a much more general class of sets $A$. On the other hand, for the case when $A$ is a Lipschitz curve and the singularity of $g(s)$ at $0^+$ is weaker than 
$s^{-\eps}$ for any $\eps>0$, Theorem~\ref{thm:mainP1} gives a more exact estimate, because~\eqref{estim-cg0}  is weaker than ~\eqref{eq:est-gC-dist}.

At the end of Section~\ref{Sec:comparison}, we give a list of open questions.  

Section~\ref{Sec:Appendix} (Appendix) contains a proof of a technical lemma, which is close to Domar's arguments.  

It should be noted that there is some resemblence between our main problem and the classical problem of describing removable sets for analytic, harmonic or suharmonic functions. The most famous of these problems is probably the Painlev\'e problem of describing removable singularities of bounded analytic functions, solved by Tolsa in \cite{Tolsa-2003}, with remarkable previous contributions by David \cite{David1998} and Melnikov and Verdera~\cite{Melnikov-Verdera}, among  others. 
We refer to Bjorn~\cite{Bjorn2002} 
for a treatment of a variant of the problem for analytic functions in $H^p$. 
The problem of removable singularities for harmonic functions comes back to Carleson~\cite{Carleson63}. The results by  
Nazarov, Tolsa and Volberg  \cite{NazarovTolsaVolberg2014} in a sense generalize the solution of the Painlev\'e problem on analytic functions to harmonic functions in $n$ dimensions. 
We refer to reviews~\cite{NazarovTolsaVolberg2014} and 
~\cite{Tolsa-review2006} for an account of these fields. 

See Gardiner \cite{Gardiner1991} and Riihentaus  \cite{Riihentaus2000} for results on the corresponding problem for subharmonic functions and 
\cite{Singman1984} for results on $n$-harmonic functions.

Despite the resemblence, the setting of all these papers differs from ours. Indeed, 
we assume {\it a priori} that our function $u$ is subharmonic (or analytic) on $A\sm B$, whereas in all these works the main question is whether this fact can be deduced from the assumptions on the growth of $u$ near~$A$. Our assumptions on $A$, typically, do not imply that $A$ is removable for the corresponding class of functions. 

\

The authors thank Istv\'an Prause and Tuomas Orponen for drawing our attention to the relation between $\p$-admissible sets and the Assouad dimension.

\section{Our main results}\label{Sec:our-main-results}

Through this paper we consider fixed a dimension $k\ge2$. Let  $p,q$ be positive integers with $p+q=k$. For any $x\in\R^k$ let us write $x=(x',x'')$ with $x'\in\R^p$ and $x''\in\R^q$. 
Given $m\in\N$, $x\in\R^m$ and $r>0$, we denote by $B_m(x,r)$ the open ball in $\R^m$ centered at $x$ with radius $r$. 
For any $x\in\R^k$ and any $r,s>0$, we define the (open) cylinder 
\[
\CC(x,r,s):=B_p(x',r)\times B_q(x'',s). 
\]
For any $a\in\R^k$, let $T_a\colon\R^k\to\R^k$ denote the translation $T_ax=x+a$. We denote by $|\cdot|$ the norm in Euclidean spaces $\R^n$ and by $m$ the Lebesgue measure in $\R^k$. 

We recall that a function $\phi: V\to \R^q$, where 
$V\subset \R^p$, is said to be {\it of Lipschitz class} if 
\[
\|\phi\|_{\text{Lip}}:=\sup_{x_1,x_2\in V, x_1\ne x_2} 
    \frac{|\phi(x_2)-\phi(x_1)|}{|x_2-x_1|}
\] 
is finite. We denote by $\mathcal{G}(\phi)$ the graph of $\phi$; it is a subset of $V\times \R^q$. 

\begin{definition}\label{def:LipschitzSurface}
A compact set $A\subset \R^k$ will be called 
\emph{a $p$-dimensional Lipschitz surface} if 
there are constants $\Li,R>0$ such that 
for any $a\in A$, there exist a Lipschitz function 
$\varphi$, defined on $B_p(0,R)$ 
and taking values in $\R^{q}$, with Lipschitz constant $\|\phi\|_{\text{Lip}}$
less or equal than $\Li$, and a $k\times k$ orthogonal matrix $U_a$ such that 
\[
\CC(0,R,3\Li R)\cap U_aT_{-a}A=\mathcal G(\varphi). 
\]
In other words, near any of its points, $A$ has to be a rotated graph of a Lipschitz function of $p$ variables. 
\end{definition}

One can replace the above parameter $3LR$ with $NLR$ for any fixed constant $N>1$, getting an equivalent definition. 

We will also work with more general sets than Lipschitz surfaces. In what follows, for a set $G\subset \R^k$ and $\si>0$, we define 
\[
[G]_\si:=\{y\in \R^k\sm G: \; 0<\dist(y,G)\le \si\}. \quad 
\]

\begin{definition}\label{def-p-adm}
	Let 
	$\p\in (0,k)$ be a real number and let $G\subset \R^k$ be a compact set. 
	Put $\q=k-\p$. 
	We say that $G$ is \emph{$\p$-admissible} if there is a constant $C$ such that for any $x\in \R^k$ and any $R, \si>0$, 
	\begin{equation}  \label{eq:p-admiss}
		m\big([G]_\si\cap B(x,R)\big)\le C\si^\q R^\p. 
	\end{equation}
\end{definition}

Any $p$-dimensional compact Lipschitz surface is $p$-admissible, because it can be represented as a finite union of (rotated) Lipschitz graphs. 
Notice also that a finite union of $p$-admissible sets is a 
$p$-admissible set. There is a relation of $p$-admissibility with the Assouad dimension, see Proposition~\ref{prop:Assouad-dim} below.  

We will need a few basic facts about 
Hardy fields of functions.   
We refer to \cite[Appendix of Chapter V]{Bourbaki-book-real-variable} for more information. 
Let $\frF$ be the collection of intervals of 
the form $[M, +\infty)$, where $M\in\R$, and 
let $\mathcal{H}(\frF, \R)$ be the ring of  
real-valued functions, defined on sets from $\frF$. 
Two functions $f_1, f_2\in \mathcal{H}(\frF, \R)$ are said to be equivalent ($f_1\sim f_2$) if they coincide on an interval belonging to $\frF$. 
A subset $\frR$ of the quotient set $\mathcal{H}(\frF, \R)/\sim$ by the above equivalence is called {\it a Hardy field} if 
it is a subfield of the ring $\mathcal{H}(\frF, \R)/\sim$ and each class in $\frR$ contains a differentiable function $f$ on an interval in $\frF$ such that the equivalence class of $f'$ 
also belongs to $\frR$ . 

Let $\frR$ be a Hardy field. In what follows, when we say that a function $f$ belongs to $\frR$ we will mean that the equivalence class of $f$ is an element of $\frR$. 
It is known that any function which represents a non-zero element of $\frR$ does not vanish on some interval of the form $[M, +\infty)$, where $M\in\R$. 
It follows that for any two functions $f_1, f_2$ in $\frR$, either $f_1<f_2$ or $f_1\ge f_2$ on some interval $[M, +\infty)$ as above. 

The field of (H) functions 
(see \cite{Bourbaki-book-real-variable}) 
is a particular example of a Hardy field. It is formed by all expressions in $t$, defined on an interval from $\frF$, that can be obtained from rational functions of $t$ by applying finitely many times the composition with $\log$, $\exp$ and arithmetical operations. 

Let us introduce the function 
$\eta\colon(0,\infty)\to\R$, given by 
\[
\eta(t)= \begin{cases}
	t^{2-k} \quad &\text{if } k>2, \\
	\log (1/t) \quad &\text{if } k=2. \\
\end{cases}
\]
Notice that $x\in\R^k\sm\{0\}\mapsto\eta(\abs x)$ is a multiple of the fundamental solution of the Laplace equation. 

Now we are ready to state our first result. Section~\ref{Sec:proof-2.4} will be devoted to its proof. 

\begin{theorem}\label{thm:mainP1}
Let $\Om\ss\R^k$ be a bounded open set. Let  
$A\subset \Om$ be a closed subset of
a Lipschitz curve (that is, a 1-dimensional Lipschitz surface) and $B$ be a closed subset of $\R^k$ such that $B\cap\Om$ is non-empty. 
Let $\psi\colon(\beta,\infty)\to(0,\infty)$ be a concave increasing function, with $\beta>0$, 
such that $\psi(t)\to\infty$ when $t\to\infty$. 
Fix some $\alpha>0$ small enough such that the composition  function $g(t)=(\psi\circ\eta)(t)$ 
is well-defined on $(0,\alpha)$. 
If $u\colon\Omega\sm B\to \R$ is a subharmonic function 
on $\Omega\sm B$ satisfying
\begin{equation}\label{eq:uleg2}
u(x)\le g(\dist(x,A\cup B)), \quad x\in\Omega, 
\; 0<\dist(x,A\cup B)<\alpha
\end{equation}
then there exist positive constants $c_1,c_2$ and  $\tau$, 
which do not depend on the function $u$, such that 
\begin{equation}\label{eq:ule2g-g}
u(x)\le 2g(c_1\dist(x,B))-g(c_2\dist(x,B)),\quad x\in\Omega,\;  0<\dist(x,B)<\tau. 
\end{equation}
\end{theorem}

\begin{remark}
	\label{rem:thm-2.4-regular-g} 
	Suppose that $g$  satisfies the hypotheses of the above theorem 
	and that for any $\beta>0$ there exists $\de>0$ such that $g(t^\beta)$ is a convex function on $(0,\de]$. Then the estimate~\eqref{eq:ule2g-g}
	implies that  
	\[
	u(x)\le g(v\dist(x,B)), \qquad x\in\Omega,\, 0<\dist(x,B)<\tau_1, 
	\]
	where $v>0$ and $\tau_1>0$ are constants. The last convexity hypothesis on $g$ is automatically fulfilled if 
	$g(1/t)$ belongs to a Hardy field that contains functions $f_\gamma(t)=t^\gamma$, $\gamma\in\R$. Indeed, in this case $(d^2/(dt)^2)g(t^\beta)$ cannot change sign infinitely many times on $(0,\al]$. 
\end{remark}

By applying quantitative Domar's theorems, 
in Section~\ref{Sec:comparison} we will get the following. 
\begin{theorem}
	\label{thm-Domar-s^-b}
	Let $0<\p<k$ and let $\Om$ be a bounded open subset of $\R^k$. DY 
	Suppose that $B\subset \R^k$ is closed and $A\subset \Om$ is a compact $\p$-admissible set such that $m(A)=0$. 
		Let $g:(0,+\infty)\to(0,+\infty)$ be a strictly decreasing function such that 
		$g(0^+)=+\infty$ 
		and 
		$g(t)\le C_1 t^{-\beta}$ for $t\in (0,1]$, where $\beta>0$ and $C_1>0$ are constants. Suppose also 
	 that $g(1/t)$ belongs to a Hardy field $\frR$, which contains all functions  $t^\gamma$, $\gamma\in\R$.   
	Then for any $a>1$ there exists a constant $v\in(0,1)$ such that the following holds.  
	For any subharmonic function $u$ on 
	$\Om\sm B$ satisfying 
	\eqref{eq:estim-u1}, one has 
	\begin{equation}
		\label{estim-cg}
		u(x)\le a g(v\dist(x,B)), \quad x\in \Om\sm B.  
	\end{equation}
\end{theorem}

The Hardy field of (H) functions, mentioned above, meets the above condition. 

Theorem~\ref{thm-Domar-s^-b} 
only applies to functions $g$ whose singularity at $0^+$ 
is at most of power-type. By using quantitative Domar's theorems, below we give Theorem~\ref{thm-Domar-A-B}, 
which applies to functions $g$ with much stronger singularity at $0^+$.

\begin{corollary}\label{cor:cor-our-theorems1}
Let $A,B,\Om$ be as in Theorem~\ref{thm-Domar-s^-b}. 

\begin{enumerate}
\item Let $g(t)=t^{-b}$, $0<b<q$. 
Then, 
by Theorem~\ref{thm-Domar-s^-b}, \eqref{eq:uleg2} implies 
that $u(x)\le Cg(\dist(x,B))$, $x\in \Om\setminus B$.  
If $A$ is a 
closed subset of a 
Lipschitz curve, 
$k\ge 3$ and 
$b\le k-2$, one can apply Theorem~\ref{thm:mainP1} to get the same conclusion. 

\item Let now $g(t)=\log^b(1/t)$ for $0<t\le 1/2$, where $0<b<\infty$. By Theorem~\ref{thm-Domar-s^-b}, 
\eqref{eq:uleg2} implies that for any $a>1$ there exist constants $C>0$ and $\tau_1>0$ such that  
\[
u(x)\le a\log^b(C/\dist(x,B)), \qquad x\in\Omega,\, 0<\dist(x,B)<\tau_1, 
\]

\item Still consider the case $g(t)=\log^b(1/t)$, where 
$0<b<\infty$ and $A$ is 
closed subset of a 
Lipschitz curve.

If $k>2$,  we can apply Theorem~\ref{thm:mainP1} to get a better estimate  
\[
u(x)\le \log^b(C/\dist(x,B)), \qquad x\in\Omega,\, 0<\dist(x,B)<\tau_1, 
\]
where $C>0$ and $\tau_1>0$ are constants. 
If $k=2$, we get this inequality if we assume additionally that $0<b\le 1$. 
\end{enumerate}
\end{corollary}

\begin{remark}
	Suppose that $g(1/t)$ belongs to a Hardy field, which is closed under the mappings $f(t)\mapsto f(Ct)$, $C>0$. If for some $b>0$, $g(t)t^b\to 0$ as $t\to 0^+$, then 
	for any $C\in (0,1)$ there are constants $C_1>0$, $t_0>0$ such that 
	\[
	g(Ct)\le C_1 g(t), \quad 0<t\le t_0. 
	\]
	Indeed, if it were not the case, then by applying the properties of Hardy fields, one would have $g(Ct)> C^{-b} g(t)$ for $0<t\le t_1$, where $t_1>0$. This implies 
	$g(C^nt_1)> C^{-bn} g(t_1)$, $n\ge 1$, which contadicts the assumption $g(t)t^b\to 0$ as $t\to 0^+$. 
	
	Now suppose that $g$ has a 
	weak singularity at $0^+$, that is, for any $\eps>0$, $g(t)t^\eps\to 0$ as $t\to 0^+$. 	Due to a similar argument, then for any numbers $C>1$, $t_0>0$ there is no $v\in (0,1)$ such that $C g(t)\le g(vt)$, $0<t\le t_0$. So for these functions $g$, the assertions of Remark~\ref{rem:thm-2.4-regular-g} and Theorem~\ref{thm-Domar-s^-b} are strictly 
	stronger than an estimate $u(x)\le C g(\dist(x,B))$, $x\in \Om\setminus B$, where $C$ is a constant. 
\end{remark}

By applying Corollary~\ref{cor:cor-our-theorems1} to the case when $k=2$ and 
$u=\log |f|$, where $f$ is analytic 
(so that $u$ is subharmonic), we get the following statement. 

\begin{corollary}\label{cor-analytic-f}
Let $\Om$ be a bounded domain in the complex plane $\C$, $B$ a closed subset of $\C$ 
and $A\subset \Om$ a compact subset of a Lipschitz curve. Let $f$ be an analytic function on $\Om\sm B$.  

\begin{enumerate}
	\item If, for some $b>0$ and $C>0$, 
	\[
	|f(z)|\le C\dist(z,A\cup B)^{-b}, \quad z\in \Om\sm (A\cup B), 
	\]
	then there is a constant $C_1$ such that 
	$|f(z)|\le C_1\dist(z,B)^{-b}$ for all $z\in \Om\sm B$. 

	\item If, for some $a,b, \alpha>0$,  
\[
|f(z)|\le 
a\bigg(\log\frac 1{\dist(z,A\cup B)}\bigg)^{b}, \quad z\in \Om\sm (A\cup B), \dist(z,A\cup B)<\alpha, 
\]
then there are constants $\eps, \tau$ 
with $\eps> \tau>0$, such that 
\[
|f(z)|\le 
a\Big(\log\frac \eps{\dist(z,B)}\Big)^{b}, \quad z\in \Om\sm B, \dist(z,B)<\tau.  
\]
\end{enumerate}
\end{corollary}

By applying Theorem~\ref{thm-Domar-s^-b} to $\log|f|$, we obtain a slightly weaker estimate than that in (2) under a weaker assumption on the size of the set $A$. 

\begin{corollary}\label{cor-analytic-f2}
	Let $\Om$ be a bounded domain in the complex plane $\C$, 
	$0<\p<2$, $B$ be a closed subset of $\C$ and $A\subset \Om$ a compact $\p$-admissible set such that $m(A)=0$. Let $f$ be an analytic function on $\Om\sm B$.  
		
      If, for some $a,b,\alpha>0$,  
		\[
		|f(z)|\le 
		a\bigg(\log\frac 1{\dist(z,A\cup B)}\bigg)^{b}, \quad z\in \Om\sm (A\cup B), \dist(z,A\cup B)<\alpha, 
		\]
		then for any $a_1>a$ there are constants $\eps, \tau$ 
		with $\eps>\tau>0$ such that 
		\[
		|f(z)|\le 
		a_1\Big(\log\frac \eps{\dist(z,B)}\Big)^{b}, \quad z\in \Om\sm B, \dist(z,B)<\tau.  
		\]
\end{corollary}

These estimates for analytic functions were, in fact, the main motivation of our research. In a forthcoming article \cite{BelloYak-resolvents}, we apply analogues of  Corollaries~\ref{cor-analytic-f} 
and~\ref{cor-analytic-f2} 
to questions related with the estimates of the resolvent growth of a function of a linear operator. 

\begin{corollary}\label{cor:union-Lip-surf}
The statements of Remark~\ref{rem:thm-2.4-regular-g}
and Corollary~\ref{cor-analytic-f} generalize to the case when $A$ is a compact subset of $\Om$ which is contained in a finite unions of Lipschitz curves.  
\end{corollary}

Indeed, one can apply induction on $n$ to prove that this statement holds if $A$ is contained in the union 
of $n$ Lipschitz surfaces.

\begin{remark}
Theorem  \ref{thm-Domar-s^-b}  
does not hold if $A$ is a general closed set. Indeed, take $B=\{x_0\}$, where $x_0\in \Om$. Fix a function $g$ 
and then a function $u$, subharmonic in $\Om\sm\{x_0\}$, whose growth at $x_0$ is such that $u$ does not satisfy the conclusion of this theorem. Then one can find a closed countable set $A\subset \Om$, whose unique accumulation point is $x_0$, such that $\dist(x,A\cup B)$ decays so rapidly when $x\to x_0$ that \eqref{eq:estim-u1} is satisfied. The same example shows that 
Theorem~\ref{thm:mainP1} and Remark~\ref{rem:thm-2.4-regular-g} also fail for a general closed set $A$. 
\end{remark}

\section{Proof of Theorem~\ref{thm:mainP1}}
\label{Sec:proof-2.4}
\noindent 
This section is devoted to the proof of Theorem~\ref{thm:mainP1}. Fix $\Omega$, $A$, $B$, $g$ and $u$ as in the formulation of that theorem. Recall that now we consider $p=1$.

Notice that the function $g$ is decreasing and $g(t)\to\infty$ as $t\to0^+$. Moreover,  since the function $x\in\R^k\sm\{0\}\mapsto\eta(\abs x)$ is a multiple of the fundamental solution of Laplace equation, the function $-g(\abs x)$ is subharmonic 
(see \cite[Theorem~2.2]{Kenn-Hayman}). 
Using that $g$ is decreasing, we can assume without loss of generality that 
$A$ is a Lipschitz curve.

In Definition~\ref{def:LipschitzSurface}, 
we can take $\Li$ as large as needed and 
$R>0$ as small as wanted. In particular, let us choose $\Li$, $R$ so that  $\Li\ge 2$ and 
$R<2\alpha$. 

Fix a domain $\DD$ in $\R^k$ with $C^2$ smooth boundary 
such that 
\[
\CC(0,1, 2\Li)\subset \DD
\subset \CC(0,1, 3\Li)
\]
Given any $a\in A$ and $r\in(0,R)$, we set 
\[
\DD_{a,r}=T_aU_a^{-1}(r\DD). 
\]
To simplify notation, we will assume that $U_a=I$. 
Then, by Definition~\ref{def:LipschitzSurface}, $\CC(a,R,3LR)\cap A$ is a graph of a function $\phi$ with $\|\phi\|_{\text{Lip}}\le L$, which is defined on the interval $(a'-R,a'+R)$ (here $a'$ is fhe first coordinate of $a$, $a=(a',a'')$).   

\begin{lemma}\label{lem:old.Claim.1}
	Let $a\in A$ and $r\in(0,R/8L)$. Assume that 
	$U_a=I$. Then for all $x\in\partial\DD_{a,r}$ we have 
	\[
	\frac{\abs{x''-\varphi(x')}}{L+1}\le\dist(x,A)\le\abs{x''-\varphi(x')}. 
	\]
\end{lemma}

\begin{proof}
	We have  
	\[
	\partial\DD_{a,r}\cap A
	=
	\partial \, \CC(a,r,2\Li r)\cap A
	=
	\big\{
	\big(a'-R,\phi(a'-R)\big), \big(a'+R, \phi(a'+R)\big)
	\big\}. 
	\]
    Fix $x\in\partial\DD_{a,r}$. 
	Then $\abs{x-a}<4\Li r$. In particular, 
	$x'\in (a'-R, a'+R)$. 
	Since $B_k(a,R)\ss\CC(a,R,3\Li R)$, the points of $A\cap B_k(a,R)$ belong to the graph of $\varphi$. 
	
	The second inequality of the statement is immediate: 
	\[
	\dist(x,A)\le\abs{x-(x',\varphi(x'))}=\abs{x''-\varphi(x')}.
	\]
	For the first inequality, notice that for some $y\in A$ 
	we have $\dist(x,A)=\abs{x-y}$. Hence 
	\[
%
%
	\abs{y-a}\le \abs{y-x}+\abs{x-a}
	\le 2\abs{x-a}<R, 
	\]
	and therefore $y=(y',\varphi(y'))$.  
    Now let $z:=(x',\varphi(x'))$. 
    Since $|x'-a'|<4Lr<R$, 
    it follows that $z\in A$. 
    Since 
    $\abs{y'-z'}=\abs{y'-x'}\le \abs{y-x}$ 
    and $\abs{z''-x''}=\abs{z-x}$, 
    we have 
    \[
    L\abs{y-x}\ge L\abs{y'-z'}\ge 
    \abs{y''-z''}
    \ge \abs{z''-x''}-\abs{y''-x''}
    \ge \abs{z-x}-\abs{y-x}. 
    \]
    This gives that 
    \[
    \abs{z-x}\le (L+1)\abs{y-x}, 
    \]
    so that 
	\[
    \frac{\abs{x''-\varphi(x')}}{L+1}
    =
    \frac{\abs{z-x}} {L+1}\le \abs{y-x}=    \dist(x,A), 
    \] 
as we wanted to prove. 
\end{proof}

Now suppose that $u$ is a subharmonic function on $\Om\sm B$, which satisfies~\eqref{eq:uleg2}. We wish to show~\eqref{eq:ule2g-g}.  
In this estimate, one can always replace $c_1$ by a smaller positive constant 
and $c_2$ by a larger constant (changing also the value of $\tau$). 
So, if estimates of the form~\eqref{eq:ule2g-g} hold on two different subsets of $\Om$ 
(for different choices of constants $c_1, c_2, \tau$), then an estimate of the same form holds also 
on the union of these two subsets.  

If $\dist(x,B)\le20\Li\dist(x,A)$, then $\dist(x,A\cup B)\ge\dist(x,B)/(20\Li)$; 
so in this case \eqref{eq:uleg2} clearly implies \eqref{eq:ule2g-g}. 
Set $\tau=R/2$. 
Then it remains to prove \eqref{eq:ule2g-g} for points $x$ satisfying 
\begin{equation}\label{eq:key.region.2}
	20\Li\dist(x,A)<\dist(x,B)<R/2.
\end{equation}
Fix a point $x_0$ as in \eqref{eq:key.region.2}, and set 
\begin{equation}\label{eq:def.r.2}
	r:=\frac{\dist(x_0,B)}{10\Li}<\frac R {20\Li}. 
\end{equation}
For some $a\in A$ we have $\dist(x_0,A)=\abs{x_0-a}<r/2$. Hence $x_0\in B_k(a,r/2)$. 
We also have 
\[
\dist(a,B)\ge\dist(x_0,B)-\abs{x_0-a}>10\Li r-r/2>9\Li r. 
\]
Therefore $B_k(a,9\Li r)\cap B=\emptyset$. 
Since  
\[
B_k(a,r)\ss\CC(a,r,2\Li r)\ss\DD_{a,r}\ss\CC(a,r,3\Li r)\ss B_k(a,4\Li r), 
\] 
we have 
\begin{equation}\label{eq:end.common.part.2}
	\DD_{a,r}\cap B=\emptyset 
	\quad \textnormal{ and } \quad 
	\dist(x_0,\partial\DD_{a,r})>r/2. 
\end{equation}
Notice also that $x_0\in \DD_{a,r}$. 

We are going to prove that $x_0$ satisfies~\eqref{eq:ule2g-g}, with some constants $c_1,c_2$ independent of $x_0$.
We may assume without loss of generality that $U_a=I$. We may also assume that 
\[
A\cap \DD_{a,R}=\mathcal{G}(\phi)=
\big\{(x_1,x''): \; x''=\phi(x'), a'-R<x'<a'+R\big\}. 
\]
The Lipschitz curve $A$ intersects $\partial\DD_{a,r}$ in two points, 
say $y$ and $z$, with $y'=a'-r$ and $z'=a'+r$. 
Consider the function $v\colon\DD_{a,r}\to\R$ given by 
\[
v(x):=g\left(\frac{\abs{x-y}}{8\Li}\right)+g\left(\frac{\abs{x-z}}{8\Li}\right).
\]
Notice that 
$\abs{a''-y''}\le Lr$ and 
$\abs{a''-z''}\le Lr$. Therefore 
for $x\in\bar \DD_{a,r}$, 
\begin{equation}\label{eq:(4L+2)r}
\max\{\abs{x-y},\abs{x-z}\}
\le 
(4L+2)r<8Lr. 
\end{equation}
Since $r<\al$, it follows that $v$ is well-defined on $\bar \DD_{a,r}$, except for the points $y,z$. 
\begin{lemma}\label{lem:estim-g-Lipsch-curve}
	For all $x\in\partial\DD_{a,r}\sm\{y,z\}$ we have 
	\[
	g(\dist(x,A))\le v(x)-g(r). 
	\]
\end{lemma}

\begin{proof}
	Take $x\in\partial\DD_{a,r}\sm\{y,z\}$. 
We distinguish three cases. First, suppose that $x'=y'$. Using Lemma~\ref{lem:old.Claim.1} and that $\varphi(x')=y''$, we have 
	\begin{equation}\label{eq:ineq-lem-g-and-v}
		\dist(x,A)\ge\frac{\abs{x''-y''}}{L+1}
		=
		\frac{\abs{x-y}}{L+1}.
	\end{equation}
	In particular, 
	\begin{equation}\label{eq:8lambda}
		\dist(x,A)\ge\frac{\abs{x-y}}{8\Li},
	\end{equation}
	and hence the statement follows using that $g$ decreases. 	
	
	Next, suppose that $y'<x'<z'$. In this case, 
\begin{align*}
	\dist(x,A)
	&\ge
	\frac{\abs{x''-\phi(x')}}{\Li+1}
    \ge
    \frac{\abs{x''-a''}-\abs{\phi(x')-a''}}{\Li+1} \\
    &\ge
    \frac{2\Li r-\Li r}{\Li+1}
	\ge 
	\frac{2r}{3}
	\ge
	\frac{\abs{x-y}}{8\Li}. 
\end{align*}
Here in the first inequality we used 
 	Lemma~\ref{lem:old.Claim.1}. In the third and fourth 
	inequalities we used that the Lipschitz constant of $\phi$ is $\Li$ and that $L\ge 2$. In the last inequality, we used that $\abs{x-y}<(4L+2)r$ 
	(see~\eqref{eq:(4L+2)r}).  
	We conclude that in this situation \eqref{eq:8lambda} also holds. 
	
	The third case $x_1=z_1$ is completely analogous to the first case $x_1=y_1$, with the roles of $y$ and $z$ interchanged. 
\end{proof}
Notice that $-v$ is subharmonic in $\DD_{a,r}$ (see \cite[Theorem~2.2]{Kenn-Hayman}), 
so $u-v$ is subharmonic in $\DD_{a,r}$. We also have 
\begin{equation}\label{eq:+infty}
\lim_{\DD_{a,r}\ni x\to y}v(x)=\lim_{\DD_{a,r}\ni x\to z}v(x)=+\infty
\end{equation}

Fix any $w\in\partial\DD_{a,r}$. 
Notice that 
\[
|x_0-w|\le |x_0-a|+|a-w|\le \Big(\frac 12+4L\Big)r \le \frac 92 Lr=\frac 9{20}\dist(x_0,B). 
\]
By~\eqref{eq:key.region.2} and the last inequality, we get  
\[
\dist(x_0,B)-\dist(x_0,A)\ge \big(1-\frac 1 {20}\big)\dist(x_0,B)\ge 2|x_0-w|. 
\]
Therefore 
\[
\dist(w,B)-\dist(w,A)\ge \dist(x_0,B)-\dist(x_0,A)-2|x_0-w|\ge 0. 
\]
We conclude that 
$\dist(w,A\cup B)=\dist(w,A)$ for any $w\in\partial\DD_{a,r}$. 
  
Combining this fact, Lemma~\ref{lem:estim-g-Lipsch-curve} and equation \eqref{eq:+infty}, we obtain
\[
\limsup_{\DD_{a,r}\ni x\to w}\, (u-v)(x)\le-g(r)
\]
for all $w\in\partial\DD_{a,r}$. Hence, by the maximum principle for subharmonic functions (see \cite[Theorem~2.2]{Kenn-Hayman}), for all $x\in\DD_{a,r}$ (in particular for $x_0$), we have 
\[
u(x)\le v(x)-g(r). 
\]

Since $\dist(x_0,\partial\DD_{a,r})>r/2$ (see~\eqref{eq:end.common.part.2}), 
we obtain $\abs{x_0-y}\ge r/2$ and $\abs{x_0-z}\ge r/2$. 
Then, using \eqref{eq:def.r.2} we get 
\begin{align*}
u(x_0)
&\le 
v(x_0)-g(r)\\
&\le
2g\left(\frac{r}{16\Li}\right)-g(r)\\
&=
2g\left(\frac{\dist(x_0,B)}{160\Li^2}\right)-g\left(\frac{\dist(x_0,B)}{10\Li}\right). 
\end{align*}
Therefore, taking $c_1:=\frac 1{160\Li^2}$, $c_2:=\frac 1 {10\Li}$ we conclude the proof of Theorem~\ref{thm:mainP1}. \qed

\section{Quantitative versions of Domar's theorems}
\label{Sec:quant-vers-Domar}

Let $\Omega$ be a bounded domain in $\R^k$, $k\ge2$ and let $F\colon\Omega\to[0,\infty]$ be a given nonnegative upper semicontinuous function.  
We recall that $\CC_F$ denotes the class of all subharmonic functions $u\colon\Omega\to \R$ 
such that $u\le F$ and define $M(x)$ by \eqref{eq:M(x)}. 

Let $\eta: (\al,\be)\to \R$ be a given decreasing function, where $(\al,\be)$ is a (finite or infinite) subinterval of $\R$.  We set    
\begin{equation}\label{eq:eta-minus}
\eta^-(s):=\inf\{t\in(\al,\be):\; \eta(t)\le s\}, 
\quad s\in (\textstyle{\lim_{\be}\eta}, \infty). 
\end{equation}
Notice that $\eta^-$ is right-continuous
and that $\eta^{-}(s)=\al$ for $s\in [\lim_{\al}\eta,\infty)$. 
If $\eta$ is continuous and strictly decreasing, then $\eta^-$ coincides 
on $(\lim_\be \eta, \lim_\al\eta)$
with $\eta^{-1}$, the inverse of $\eta$. 

We define $\eta_+$, the right-continuous regularization of the function $\eta$ by $\eta_+(t)=\lim_{t^+}\eta(t)$. 
Functions $\eta$ and $\eta_+$ differ only on a countable set, and $\eta^-(s)=(\eta_+)^-(s)$ for all $s$. 

If $\eta$ decreases and is right-continuous, then  $\eta(\eta^-(s))\le s$ for all $s$. 

Recall that given a measurable function $H\colon\R\to[0,\infty]$, 
if $h\colon\Om\to[0,\infty]$ is its distribution function, 
i.e. $h(s):=m(\{H>s\})$, then for any $t>0$ we have 
\begin{equation}\label{eq:distrib.funct}
	\int_{\{H>t\}}H(x)\, dx=\int_{t}^{\infty}h(s)\, ds+th(t).
\end{equation}

Denote by $f\colon[0,\infty)\to[0,\infty]$ the distribution function of $F$. 
Then $f^-$ is equidistributed with $F$, that is, for any $y_0>0$, the one-dimensional measure 
of the set where $f^-(s)>y_0$ is the same as the $k$-dimensional measure of the subset of $\Om$, where $F(x)>y_0$. 
Similarly, $\log^+ f^-$ is equidistributed with $\log^+F$. 

In this section we will obtain a quantitative version of the following results by Domar \cite{Domar} 
using a slight refinement of his arguments. 
We will use a notation quite similar to that  employed in \cite{Domar} to facilitate the comparison. 

\begin{theoremA}[{\cite[Theorem 2]{Domar}}]
	If for some $\eps >0$, 
	\begin{equation}
		\label{eps-cond}
		\int_\Om 
		\big[\log^+ F(x)\big]^{k-1+\eps}\, dx <\infty, 
	\end{equation}
	then $M(x)$ is bounded on every compact subset of $\Om$. 
\end{theoremA}

\begin{theoremB}[{\cite[Theorem 3]{Domar}}]
	Suppose that 
	$F(x_1,\dots, x_k)$ only depends on the first $\q$ variables, where $1\le \q\le k-1$. If 
	\begin{equation}
		\label{q-cond}
		\int_0^{\abs\Om}
		\log^+ f^-(s)d (s^{1/\q}) <+\infty, 
	\end{equation}
	then $M(x)$ is bounded on every compact subset of $\Om$. 
\end{theoremB}

We remark that in his statements, Domar understood that subharmonic functions 
	cannot take value $-\infty$. 
	Here we allow this value. This does not affect Domar's statements. Indeed, 
	$\max(u,0)$ is a subharmonic function for any subharmonic function $u$, so that 
	$M(x)$ does not depend on whether in~\eqref{eq:M(x)}, suharmonic functions that take value $-\infty$ are allowed or not.  

As Domar mentions in \cite{Domar}, if $\q>1$, then~\eqref{q-cond} holds if 
\begin{equation}\label{w-cond}
\int_\Om 
\big[\log^+ F(x)\big]^w\, dx <\infty 
\end{equation}
for some $w<\q$. If $\q=1$, then~\eqref{q-cond} is equivalent to condition 
\eqref{w-cond} for $w=1$. 

These results by Domar \cite{Domar} are vast generalizations of Levinson's 
$\log\log$ theorem 
\cite{Levinson} on analytic functions. 
It is worth mentioning that in~\cite{Domar1988}, Domar has generalized his results in \cite{Domar} to solutions of certain elliptic and parabolic equations. 
In \cite{Logunov}, Logunov gave an analogue of Levinson's $\log\log$ theorem for harmonic functions. Logunov and Papazov in 
\cite{LogunPapaz} give a version of Levinson's theorem for 
solutions of elliptic equations and relate this result with what they call a three ball inequality.

\subsection{Our version of Domar's Theorem A}
From now on, we fix a constant $a>1$. 

Denote by $S_R$ the volume of the $k$-dimensional ball $B_k(0,R)$. 
Given $u\in\CC_F$, for any $\nu\in\R$ set 
\[
E_\nu:=\{a^\nu\le u<a^{\nu+1}\}\ss\Omega
\quad\tn{ and }\quad
\ell_\nu:=m(E_\nu). 
\]
Notice that if \eqref{eps-cond} holds, then $\ell_\nu$ is finite for all $\nu>0$. 

\begin{lemma}[See {\cite[Lemma 1]{Domar}}]\label{lem:Domar.L1}
	Let $D$ be a positive constant and $\lambda$ a positive integer, both so large that 
	\begin{equation}\label{eq:D.lambda}
		\frac{a}{D^kS_1}+\frac{1}{a^\lambda}\le1.
	\end{equation}
	Suppose that for some $\nu>\lambda$ and some $x_\nu\in\Omega$ we have 
	\[
	u(x_\nu)\ge a^\nu
	\quad\tn{ and }\quad
	B_k(x_\nu,R)\ss\Omega,
	\]
	where 
	\[
	R>D(\ell_{\nu-\lambda}+\dotsb+\ell_\nu)^{1/k}.
	\]
	Then there exists $x_{\nu+1}\in B_k(x_\nu,R)$ such that 
	\[
	u(x_{\nu+1})\ge a^{\nu+1}.
	\]
\end{lemma}

Domar used the value $a=e$, but here we will make some advantage of choosing finally~$a$ close to ~$1$. 

As an immediate consequence of the above lemma we obtain the following result. 

\begin{corollary}\label{cor:Domar.L1}
	Let $D$ and $\lambda$ be as in Lemma~\ref{lem:Domar.L1}. 
	Suppose that $u(x_t)\ge a^t$ for some $u\in\CC_F$ and some $x_t\in\Omega$, 
	where $t>\lambda$. Let $s$ be a real number such that $s-t$ is a positive integer and let 
	\[
	R>D\sum_{n=0}^{s-t-1}(\ell_{t+n-\lambda}+\dotsb+\ell_{t+n})^{1/k}. 
	\]
	Then either $B_k(x_t,R)$ intersects the boundary of $\Omega$ or 
	$u(x_s)\ge a^s$ for some $x_s\in B_k(x_t,R)$. 
\end{corollary}

We set 
\[
f_1(t)=f(a^t), \quad t>0. 
\]

\begin{lemma}\label{lem:dist.delta}
	Let $D$ and $\lambda$ be as in Lemma~\ref{lem:Domar.L1}. 
	For $t\in(\lambda+1,\infty)$, let 
	\[
	\delta(t):=(\lambda+1)\pa{\frac{k-1}{\epsilon}}^{\frac{k-1}{k}}\frac{1}{(t-1-\lambda)^{\epsilon/k}}
	\]
	and 
	\[
	\psi(t)
	:=
	\co{
		\int_{(t-\lambda)^{k-1+\epsilon}}^{\infty}f_1(s^\frac{1}{k-1+\epsilon})\, ds+(t-\lambda)^{k-1+\epsilon}f_1(t-\lambda)
	}^{\frac1k},
	\]
	and set $\varphi(t):=\delta(t)\psi(t)$. 
	If for some $u\in\CC_F$, $t>\lambda+1$ and $x_t\in\Omega$, one has  
	$u(x_t)\ge a^t$, then 
	\[
	\dist(x_t,\partial\Omega)<D\varphi(t). 
	\]
\end{lemma}

This lemma is just a slight refinement of the computations given in the proof of Theorem~2 of \cite{Domar}. We postpone its proof to the Appendix. 

Notice that $\delta(t)\to0$ when $t\to\infty$, and \eqref{eps-cond} implies that $\psi$ is bounded: 
\begin{align*}
0
\le
\psi(t)
&=
\co{\int_{\{\log^+F> t-\lambda\}}[\log^+F(x)]^{k-1+\epsilon}\, dx}^{\frac{1}{k}}\\
&\le 
\co{\int_{\Omega}[\log^+F(x)]^{k-1+\epsilon}\, dx}^{\frac{1}{k}}
<
\infty.
\end{align*}
Consequently $\varphi(t)\to0$ when $t\to\infty$. Now Theorem~A can be deduced from Lemma~\ref{lem:dist.delta} as follows. 
Suppose that $M$ is not bounded on some compact set $K\ss\Omega$ and put 
$\rho:=\dist(K,\partial\Omega)>0$. 
Take $t\gg0$ so that $D\varphi(t)<\rho$. Let $x_t\in K$ be such that $M(x_t)>a^t$. 
Then, for some $u\in\CC_F$ we have $u(x_t)\ge a^t$, and by Lemma~\ref{lem:dist.delta} we arrive to the contradiction
$
\dist(x_t,\partial\Omega)<D\phi(t)<\rho. 
$

Let us see how Lemma~\ref{lem:dist.delta} can be used to obtain a quantitative version of Theorem~A. 
First, notice that the function $\delta\colon(\lambda+1,\infty)\to\R$ is positive, strictly decreasing and 
$\delta(t)\to\infty$ when $t\to(\lambda+1)^+$. 
On the other hand, the function $\psi\colon(\lambda+1,\infty)\to\R$ is decreasing. 
Moreover, if 
\begin{equation}\label{eq:extra.cond.F}
	m(\{x\in\Omega:F(x)>s\})>0
\end{equation}
for all $s$, 
then we can also guarantee that $\psi$ is positive. 
Consequently $\varphi\colon(\lambda+1,\infty)\to\R$ is a positive strictly decreasing function 
such that $\varphi(t)\to\infty$ when $t\to(\lambda+1)^+$ and $\varphi(t)\to0$ when $t\to\infty$. 
Also, since $\delta$ is continuous and 
$\psi$ is right-continuous 
the product $\varphi=\delta\psi$ is right-continuous.

We will use the right-continuous function $\varphi^-\colon(0,\infty)\to(\lambda+1,\infty)$ 
(see~\eqref{eq:eta-minus}).  

\begin{theorem}[A quantitative version of Theorem~A]
	\label{thm-Domar-2-quant}
Let $\Omega$ be a bounded domain in $\R^k$, $k\ge2$, and let 
$F\colon\Omega\to[0,\infty]$ be a nonnegative upper semicontinuous function such that 
$m(\{x\in\Omega:F(x)>s\})>0$ for all $s$, and 
\[
\int_\Om\big[\log^+ F(x)\big]^{k-1+\eps}\, dx <\infty
\]
for some $\epsilon>0$.  
Let $\CC_F$ be the class of all subharmonic functions $u\colon\Omega\to \R$ 
such that $u\le F$ and set 
$M(x):=\sup\{u(x):u\in\CC_F\}$. 	
Fix a constant $a>1$ and let the function $\varphi$ be as in Lemma~\ref{lem:dist.delta}. 
	Then for all $x\in\Omega$ we have 
	\[
	M(x)\le a^{\varphi^{-}(D^{-1}\dist(x,\partial\Omega))}.
	\]
\end{theorem}

\begin{proof}
	Suppose on the contrary that $M(x)>a^{\varphi^{-}(D^{-1}\dist(x,\partial\Omega))}$ 
	for some $x\in\Omega$. Then, for some $u\in\CC_F$ we have 
	$u(x)>a^{\varphi^{-}(D^{-1}\dist(x,\partial\Omega))}$. 
	Since $\varphi^{-}(D^{-1}\dist(x,\partial\Omega))>\lambda+1$, 
	by Lemma~\ref{lem:dist.delta} we arrive to a contradiction:  
	\[
	\dist(x,\partial\Omega)
	<
	D\varphi(\varphi^{-}(D^{-1}\dist(x,\partial\Omega)))
	\le
	\dist(x,\partial\Omega). 
	\qedhere
	\]
\end{proof}

\subsection{Our version of Domar's Theorem B}

\noindent
Let $F_\nu:=\{x\in\Omega: F(x)\ge a^\nu\}$
for each $\nu>0$. Choose some $\p\in(0,k)$ and put $\q=k-\p$. Notice that, to the contrary to Domar's Theorem~B, we allow non-integer values of $\q$.  
Set  
\begin{equation}\label{eqn:def-mu}
	\mu_\q(\nu):=\sup\,\frac{m(F_\nu\cap B_k(x,R))}{R^\p}, \qquad \nu>0, 
\end{equation}
where the supremum is taken over all balls $B_k(x,R)$ contained in $\Omega$. 
Notice that $\mu_\q$ is a decreasing function, that $\mu_\q(0^+)$ exists and is positive.

\begin{lemma}\label{lem:Domar.L2}
	Let $D$ be a positive constant and $\lambda$ a positive integer, both so large that 
	\begin{equation}\label{eq:D.lambda.2}
		\frac{a}{D^\q S_1}+\frac{1}{a^\lambda}\le1.
	\end{equation}
		Let $u\in\CC_F$. Suppose that for some $\nu>\lambda$ and some $x_\nu\in\Omega$ we have 
	\[
	u(x_\nu)\ge a^\nu
	\quad\tn{ and }\quad
	B_k(x_\nu,R)\ss\Omega,
	\]
	where 
	\[
	R>D[\mu_\q(\nu-\lambda)]^{1/\q}.
	\]
	Then there exists $x_{\nu+1}\in B_k(x_\nu,R)$ such that 
	\[
	u(x_{\nu+1})\ge a^{\nu+1}.
	\]
\end{lemma}

\begin{proof}
	The proof is quite similar to that of Lemma~\ref{lem:Domar.L1}. 
	If $u(x)<a^{\nu+1}$ for all $x\in B_k(x_\nu,R)$ then 
	\begin{align*}
		u(x_\nu) 
		&\le 
		\frac{1}{S_R}\int_{B_k(x_\nu,R)}u(x)\, dx\\
		&= 
		\frac{1}{S_R}\co{\int_{F_{\nu-\lambda}\cap B_k(x_\nu,R)}u(x)\, dx
			+
			\int_{B_k(x_\nu,R)\sm F_{\nu-\lambda}}u(x)\, dx}\\
		&\le 
		\frac{1}{S_R}(a^{\nu+1}\mu_\q(\nu-\lambda)R^{\p}+a^{\nu-\lambda}S_R)\\
		&= 
		a^\nu\co{\frac{a\mu_\q(\nu-\lambda)R^{\p}}{R^kS_1}+\frac{1}{a^\lambda}}\\
		&< 
		a^\nu\co{\frac{a}{D^\q S_1}+\frac{1}{a^\lambda}}\\
		&\le 
		a^\nu,
	\end{align*}
	which contradicts the hypothesis $u(x_\nu)\ge a^\nu$. 
\end{proof}

\begin{corollary}\label{cor:Domar.L2}
	Let $D$ and $\lambda$ be as in Lemma~\ref{lem:Domar.L2}. 
	Suppose that $u(x_t)\ge a^t$ for some $u\in\CC_F$ and some $x_t\in\Omega$, 
	where $t>\lambda$. Let $s$ be a real number such that $s-t$ is a positive integer and let 
	\[
	R>D\sum_{n=0}^{s-t-1}[\mu_\q(t+n-\lambda)]^{1/\q}. 
	\]
	Then either $B_k(x_t,R)$ intersects the boundary of $\Omega$ or 
	$u(x_s)\ge a^s$ for some $x_s\in B_k(x_t,R)$. 
\end{corollary}

\begin{lemma}\label{lem:dist.rho}
	Let $D$ and $\lambda$ be as in Lemma~\ref{lem:Domar.L2}. 
	Suppose that 
	$\mu_\q(\nu)$ is finite for any $\nu>0$ and 
		\begin{equation}\label{q-cond2}
		\int_0^1\mu_\q^-(s)\, ds^{1/\q}<\infty.  
	\end{equation}
	Set 
    \begin{equation}\label{def-rho}
	\rho(t):=\int_0^{\mu_\q(t-\lambda)}(\mu_\q^-(s)-t+1+\lambda)\, ds^{1/\q}, \qquad t>\lambda. 
    \end{equation}
	If for some $u\in\CC_F$, $t>\lambda$ and $x_t\in\Omega$ one has $u(x_t)\ge a^t$, 
	then 
	\[
	\dist(x_t,\partial\Omega)<D\rho(t). 
	\]
\end{lemma}

\begin{proof}
	Since $u$ is bounded from above on compact subsets of $\Omega$, 
	by Corollary~\ref{cor:Domar.L2} we have 
	\begin{align*}
		\dist(x_t,\partial\Omega) 
		&\le 
		D\sum_{n=0}^{\infty}[\mu_\q(t-\lambda+n)]^{1/\q}\\
		&= 
		D\sum_{n=0}^{\infty}\int_0^{\mu_\q(t-\lambda+n)}ds^{1/\q}\\
		&= 
		D\sum_{n=0}^{\infty}\int_{\mu_\q(t-\lambda+n+1)}^{\mu_\q(t-\lambda+n)}(n+1)\, ds^{1/\q}\\
		&<  
		D\sum_{n=0}^{\infty}\int_{\mu_\q(t-\lambda+n+1)}^{\mu_\q(t-\lambda+n)}
		\big(\mu_\q^-(s)-t+1+\lambda\big)\, ds^{1/\q} \\
		&= 
		D\int_0^{\mu_\q(t-\lambda)}
		\big(\mu_\q^-(s)-t+1+\lambda\big)\, ds^{1/\q} \\
		&= 
		D\rho(t).
		\qedhere
	\end{align*}
\end{proof}

Notice that the function $\rho$ is decreasing. 

\begin{theorem}[A quantitative version of Theorem~B]
	\label{thm-Domar-3-quant} 
Let $\Omega$ be a bounded domain in $\R^k$, $k\ge2$, and let 
$F\colon\Omega\to[0,\infty]$ be a nonnegative upper semicontinuous function. 
Fix a constant $a>1$ and let $F_\nu:=\{x\in\Omega: F(x)\ge a^\nu\}$
for each $\nu>0$. Choose some $\p\in(0,k)$ and put $\q=k-\p$. 
Set  
\[
	\mu_\q(\nu):=\sup\,\frac{m(F_\nu\cap B_k(x,R))}{R^\p}, \qquad \nu>0, 
\]
where the supremum is taken over all balls $B_k(x,R)$ contained in $\Omega$. 
Suppose that $\mu_\q(\nu)$ is finite for any $\nu>0$ and 
\[
\int_0^1\mu_\q^-(s)\, ds^{1/\q}<\infty 
\]
(see~\eqref{eq:eta-minus}). 
Define the function $\rho$ as in Lemma~\ref{lem:dist.rho}. 
	Then for all $x\in\Omega$ we have 
	\[
	M(x)\le a^{\rho^{-}(D^{-1}\dist(x,\partial\Omega))}.
	\]
\end{theorem}

We remark that condition~\eqref{q-cond} of Theorem B implies condition~\eqref{q-cond2}.

\begin{proof}[Proof of Theorem~\ref{thm-Domar-3-quant}]
	Suppose on the contrary that 
	\[
	M(x)> a^{\rho^{-}(D^{-1}\dist(x,\partial\Omega))}
	\] 
	for some $x\in\Omega$. Then there are some $u\in\CC_F$ and $\kappa_0\in\R$ such that 
	\[
	u(x)>a^ {\kappa_0}>a^{\rho^{-}(D^{-1}\dist(x,\partial\Omega))}. 
	\] 
	Since $\rho^{-}(D^{-1}\dist(x,\partial\Omega))>\lambda$, 
	by Lemma~\ref{lem:dist.rho} we arrive to a contradiction:  
	\[
	\dist(x,\partial\Omega)<D\rho(\kappa_0)\le D\rho(\rho^{-}(D^{-1}\dist(x,\partial\Omega)))
	\le
	\dist(x,\partial\Omega). \qedhere
	\]
\end{proof}

\begin{remark}\label{rem:mu_\q}
	Theorem~\ref{thm-Domar-3-quant} 
	is still true if, instead of defining the function $\mu_\q$ by 
	\eqref{eqn:def-mu}, we just require that 
	\[
	\sup_{B_k(x,R)\subset \Om}\frac{m(F_\nu\cap B_k(x,R))}{R^\p}\le \mu_\q(\nu) 
	\]
	for all $\nu>0$. This can be seen easily from the proof of this theorem. This observation will be useful for application of this result.  
\end{remark}

\section{Application of quantitative Domar's results to our problem. Proof of Theorem~\ref{thm-Domar-s^-b}}
\label{Sec:comparison}

Here we show how to apply Theorem~\ref{thm-Domar-3-quant} 
to obtain an answer to our main question for the case when $A$ is a $\p$-admissible set. We are going to prove the following. 
\begin{theorem}\label{thm-Domar-A-B}
	Let $0<\p<k$ and let $\Om$ be a bounded open subset of $\R^k$. 
	Suppose $g$ is continuously differentiable. 
	Fix $a>1$. 
	Choose $D>1$ so that this value of $D$ and $\lambda=1$ satisfy \eqref{eq:D.lambda.2}. 
Suppose that $B\subset \R^k$ is closed and $A\subset \Om$ is a compact $\p$-admissible set such that $m(A)=0$.  
Let $C_1=\max(C,1)$, where $C$ is the constant in the condition~\eqref{eq:p-admiss}
for $\p$-admissibility. Put
\[
\mu_{ad}(\nu):= 
C_1\big(g^{-1}(a^\nu)\big)^\q
\]
and let 
\begin{equation}\label{eq-expr-mu-}
	\mu_{ad}^-(t)=\log_a g\Big(\Big(\frac t {C_1}\Big)^{\oneq}\Big) 
\end{equation}
be its inverse function. Assume that $\mu_{ad}^-$ satisfies~\eqref{q-cond2}. Put 
\begin{equation}\label{eq-def-rho-ad}
	\rho_{ad}(t)=\int_0^{\mu_{ad}(t-1)}(\mu_{ad}^-(s)-t+2)\, ds^{1/\q}, \qquad t>1. 
\end{equation}
Set $\tau=\dist(A,\partial\Om)/2>0$. 
Then for any subharmonic function $u$ on $\Om\sm B$ satisfying~\eqref{eq:estim-u1}, one has 
\begin{equation}
	\label{estim-cg-}
	u(x)\le a^{\rho_{ad}^{-}\big(\dist(x,B)/(3D)\big)}, 
	\quad x\in \Om\sm B, \dist(x,B)<\tau.  
\end{equation}
\end{theorem}

Notice that~\eqref{eq-def-rho-ad} is just the formula~\eqref{def-rho}, applied to $\mu_{ad}$ in place of $\mu_\q$ and to $\la=1$. 

Let us relate the notion of a $p_*$-admissible set
with the Assouad dimension. 
If $H\subset \R^k$ 
is a bounded set and $r>0$, we denote by $N_r(H)$ the smallest number of open balls of radius $r$ whose union contains $H$.

\begin{definition}
	Let $A$ be a bounded non-empty subset of $\R^k$. 
	Given $p\in (0,k)$ and $C>0$, consider the condition 
	\begin{enumerate}
		\item[$(\mathcal{A}(p,C))$] \qquad 
		for all $0<r<R$ and $x\in A$, 
		$\displaystyle N_r(B(x,R)\cap A)\le C\Big(\frac R r\Big)^p$. 	
	\end{enumerate}
	The {\it Assouad dimension of} $A$ is defined as 
	\[
	\dim_{\text{As}} A = \inf 
	   \Big\{ p: \text{there exists $C>0$ such that $(\mathcal{A}(p,C))$ holds} \Big\}.     
     \]
\end{definition}

We refer to the book by Fraser~\cite{Fraser-book} for a comprehensive account of this notion.

\begin{proposition}\label{prop:Assouad-dim}
Let $\p<k$. If the Assouad dimension of $A$ is less than $\p$, then $A$ is $\p$-admissible. 
\end{proposition}
\begin{proof} 
By the assumption, $(\mathcal{A}(p_*,C))$ holds for some constant $C>0$. Hence for all $0<\si<R$, $B(x,R)\cap A$ can be covered by less than $C(R/\si)^\p$ balls of radius $\si$. It follows that $[A]_\si$ is covered by the balls with the same centers of radia $2\si$. This implies~\eqref{eq:p-admiss}, which shows that $A$ is $\p$-admissible. 
\end{proof}	

There are many examples of sets $A\subset \R^k$ of non-integer Assouad dimension, in particular, self-similar sets that are attractors of iterated function systems, which satisfy the so-called Open Set Condition. See \cite{Fraser-book}, Theorems 6.4.1, 6.4.3 and Corollary 6.4.4 and Falconer's book \cite{Falconer-book}, Chapter 9 (in particular, Theorem 9.3). There are cases when $\dim_{\text{As}} A$ is strictly between $k-1$ and $k$.

\begin{proof}[Proof of Theorem \ref{thm-Domar-A-B}]
Integrating by parts, we get 
\begin{equation}
	\label{eq:expression-rho}
	\begin{aligned}
		\rho_{ad}(\nu)&=\int_0^{\mu_{ad}(\nu-1)}
		\big(
		\mu_{ad}^-(x)+2-\nu
		\big)\, dx^{\oneq} \\
		& = \mu_{ad}^-(x) x^{\oneq}\Big|_{x=0}^{x=\mu_{ad}(\nu-1)} 
		 - \int_0^{\mu_{ad}(\nu-1)} (\mu_{ad}^-)'(x)x^{\oneq}\, dx 
		+\big(2-\nu\big)\mu_{ad}(\nu-1)^{\oneq} \\
		& = (\nu-1)\mu_{ad}(\nu-1)^{\oneq}
		+ (2-\nu)\mu_{ad}(\nu-1)^{\oneq} 
		- \int_0^{\mu_{ad}(\nu-1)} (\mu_{ad}^-)'(x)x^{\oneq}\, dx \\
		& = \mu_{ad}(\nu-1)^{\oneq}   
		- \int_0^{\mu_{ad}(\nu-1)} (\mu_{ad}^-)'(x)x^{\oneq}\, dx\, .  
	\end{aligned}
\end{equation} 
Since $(\mu_{ad}^-)'\le 0$, it follows that 
	$\rho_{ad}(\nu)\ge \mu_{ad}(\nu-1)^{\oneq}$.
Hence, by~\eqref{eq-expr-mu-}, 
\begin{equation}\label{estim-g-rho-}
\rho_{ad}^-(t)\ge 1+\log_a g(C_1^{-\oneq}t)\ge 1+\log_a g(t). 
\end{equation}
	
	Suppose $u$ satisfies \eqref{eq:estim-u1}, and fix a point $x\in \Om\sm A$ such that $0<\dist(x,B)<\tau$. 
	Consider two cases. 

\smallskip 

	{\noindent\scshape Case 1:} $\dist(x,A)\ge \dist(x,B)/3$. Then, by 
	~\eqref{estim-g-rho-}, 
		\begin{align*}
		u(x) 
		&\le 
		g(\dist(x,A\cup B))
		\le 
		g(\dist(x, B)/3)\\ 
		&\le 
		a^{-1+\rho^-_{ad}(\dist(x, B)/3)}\le 
		a^{\rho^-_{ad}(\dist(x, B)/(3D))}, 
		\end{align*} 			 
which gives \eqref{estim-cg-} in this case. 	 
	
\smallskip 
	
	{\noindent\scshape Case 2:} $\dist(x,A)< \dist(x,B)/3<\tau$. Put $d=\dist(x,B)/3$. Then 
	\[	
     \dist(x,\partial\Om)\ge \dist(A,\partial\Om)-\dist(x,A)= 2\tau-\dist(x,A)>\tau>d. 
     \]
    Hence $B(x,d)\subset \Om$. 
    We apply the estimate of Theorem~\ref{thm-Domar-3-quant} to the point $x$ and the open ball 
	$B(x,d)$. Notice that 
	for any $y\in B(x,d)$, 
	\[
	\dist(y,A)\le d+\dist(x,A)\le 2d=
	\dist(x,B)-d\le \dist(y,B). 
	\]
	Hence we have $u\le F$ in $B(x,d)$, where 
		the majorant $F$ has the form $F=g\big(\dist(\cdot,A)\big)$. 
	We put $F(y)=+\infty$ if $y\in A$. This function is upper semicontinuous on $B(x,d)$. We have 
    \[
    F_\nu=\{y\in B(x,d): \; F(x)\ge a^\nu\}
      = \{y\in B(x,d): \; \dist(y,A)\le\s(\nu)\}, 
    \]
    where we have denoted $\s(\nu)=g^{-1}(a^\nu)$. 
    Since $A$ is $\p$-admissible,  
     \[
	 \sup_{B(y,R)\subset B(x,d)}\frac{m\big(F_\nu\cap B_k(x,R)\big)}{R^\p}
     \le  C\s(\nu)^\q= C\big(g^{-1}(a^\nu)\big)^\q\le \mu_{ad}(\nu). 
     \]
     By Theorem~\ref{thm-Domar-3-quant}, 
     applied to the ball $B_k(x,d)$ and its center 
     (see also Remark~\ref{rem:mu_\q}), 
\[
	u(x)\le a^{\rho_{ad}^{-}(d/D)} 
= a^{\rho_{ad}^{-}(\dist(x,B)/(3D))}, 
	\quad x\in \Om\sm B, \dist(x,B)<\tau.  \qedhere
\]
\end{proof}

\begin{proof}[Proof of Theorem \ref{thm-Domar-s^-b}]
We will use~\eqref{eq:expression-rho}. 
First we observe that the limit  
\begin{equation}
	\label{eq:asymp-int}
L:=\lim_{t\to 0^+}
\frac{\int_0^t (\mu_{ad}^-)'(x)x^{\oneq}\, dx}
  {t^{\oneq}} 
\end{equation}
exists and satisfies $-\infty<L\le 0$.     
Indeed, by the L'Hospital rule,  
    \begin{equation}
    	\label{eq:L}
    L=\q \lim_{t\to 0^+}\frac{(\mu_{ad}^-)'(t)t^{\oneq}}{t^{\oneq-1}}=
\q \lim_{t\to 0^+}t (\mu_{ad}^-)'(t). 
\end{equation}
Since $g(1/s)$, $(d/ds)g(1/s)$ and all power functions 
$s^b$ belong to $\frR$, it is easy to get from 
\eqref{eq-expr-mu-} that the last limit exists, but might be infinite.  
Since $(\mu_{ad}^-)'(t)\le 0$, we have 
$-\infty\le L\le 0$. It remains to see that $L\ne -\infty$. 

Indeed, if $\lim_{t\to 0^+}t (\mu_{ad}^-)'(t)=-\infty$, 
then for any $\gamma>0$ there would exist some $t_0=t_0(\gamma)>0$ such that $(\mu_{ad}^-)'(t)<-\gamma/t$ for $0<t\le t_0$. By integrating between $t$ and $t_0$, we would get that for any $\gamma>0$, $a^{\mu_{ad}^-(t)}>K(\gamma) t^{-\gamma}>0$ 
for $0<t<t_0(\gamma)$. This, together with the growth assumption on $g$, contradicts~\eqref{eq-expr-mu-}.

Notice that $\mu_{ad}(\nu-1)\to 0$ as $\nu\to\infty$. Since $L>-\infty$,~\eqref{eq:expression-rho} implies that  
\[
\rho_{ad}(\nu)\le C_2 \mu_{ad}(\nu-1)^{\oneq}, 
\]
say, for $\nu\ge \nu_0$ (in particular,~\eqref{q-cond2} holds). 
By~\eqref{eq-expr-mu-}, it follows that 
\[
\rho_{ad}^-(t)\le 1 +  \mu_{ad}^-(C_2^{-\q}t^\q)
\le 1 + \log_a g(v_0 t),   
\]
say, for $0<t\le t_0$, where $t_0>0$ and $v_0>0$ are  constants. Now we apply Theorem~\ref{thm-Domar-A-B} and obtain an estimate
\begin{equation}
	\label{eq:est-M(x)}
u(x)\le 
a^{ \rho_{ad}^-(\dist(x, B)/(3D))}
 \le a  g\big(v\dist(x, B)\big), 
 \quad 
 \dist(x,B)\le \tau, 
\end{equation}
where $v=v_0/(3D)$ and $\tau>0$ is a constant. By substituting $v$ with a smaller positive constant, we get that~\eqref{eq:est-M(x)} holds for any $x\in\Om\sm B$. Hence~\eqref{estim-cg} is valid. 
\end{proof}

\begin{remark}\label{rem:worse-s^-b}
	If $g(t)t^\beta\to +\infty$ as $t\to0^+$ for any $\beta>0$, but the 
	rest of assumptions of Theorem~\ref{thm-Domar-s^-b} hold, 
	then the quantitative Domar's Theorem~\ref{thm-Domar-3-quant}
    does not give the estimate $u(x)\le g(v\dist(x,B))$, where $v>0$ is a constant.  
    
    Indeed, 
	the arguments of the above proof can be reversed, and one gets that in this case, $L=-\infty$. Hence now~\eqref{eq:expression-rho} gives that for any $C_1>0$ there exists $\nu_0$ such that 
	$\rho(\nu)\ge C_1 \mu_\q(\nu-1)^{\oneq}$ for $\nu\ge \nu_0$. This in its term implies that for any $v_0>0$ there exists $t_0>0$ such that 
	\[
	\rho^-(t)\ge 1 +\log_a g(v_0t), \quad 0<t\le t_0. 
	\] 
	Hence for any $v>0$ there exists $C>0$ such that the majorant provided by Theorem~\ref{thm-Domar-3-quant} satisfies  
	\[
	a^{\rho^-(\dist(x,B)/D)}\ge C g(v\dist(x,B)), 
	\]
    whenever $\dist(x,B)$ is sufficiently small. 
    Since $g(t)$ grows more rapidly than any power $t^{-\beta}$, it is easy to get that for any $v_1, v_2$
    such that $0<v_1<v_2$, $\limsup_{t\to0^+} g(v_1t)/g(v_2t)= +\infty$. This implies our  assertion.  		
\end{remark}

\begin{remark}\label{rem:about-Thm-5.4}
Theorem~\ref{thm-Domar-A-B} has been obtained by applying 
Theorem~\ref{thm-Domar-3-quant} to our main problem. Instead one could try to apply Theorem~\ref{thm-Domar-2-quant}. 
It looks like, however, that this gives worse estimates. For instance, under the hypotheses of Theorem~\ref{thm-Domar-s^-b}, by applying Theorem~\ref{thm-Domar-2-quant} one only gets that 
\[
u(x)\le \big(\dist(x,A\cup B)\big)^{-\beta}, \; x\in\Om\sm (A\cup B)\;
\]
implies
\[
u(x)\le \eta(\dist(x,B)), \; x\in\Om\sm B, 
\]
where $\eta(y)=\eta_1(y)y^{-\frac{k\beta}\q}$ 
and $\lim_{0^+}\eta_1=+\infty$. 
This is worse than the estimate given by 
Theorem~\ref{thm-Domar-s^-b} (which is a consequence of Theorem~\ref{thm-Domar-A-B}). 
\end{remark}

As the next example shows, if the growth of $g$ is close to the limit one permitted in Domar's estimates, Theorem~\ref{thm-Domar-3-quant}
in fact gives a much worse estimate than $g(v\dist(x,B))$. 

\begin{example}\label{ex:fast-growth-g}
	Consider the case of the function $g(t)=\exp(t^{-\alpha})$, where $\al>0$. It grows at $0^+$ faster than any power $t^{-\beta}$. Assume that $A$ is $\p$-admissible for some $\p<k$. Then we get 
	\[
	\si(\nu)=\nu^{-\frac 1 \al},
	\quad 
	\mu_{ad}(\nu)=C'\nu^{-\frac \q \al}, 
	\quad 
    \mu_{ad}^-(x)=Cx^{-\frac \al \q}
	\] 
	(where $C', C>0$ are some constants). 
	To meet the condition~\eqref{q-cond2}, 
	we have to impose the condition $\al<1$. 

	An easy calculation using~\eqref{eq:expression-rho} yields that $\rho_{ad}^-(t)\sim C_2 t^{-\frac \al{1-\al}}$ as $t\to 0^+$, where 
	$C_2>0$ is a constant. 
    We get that the majorant from Theorem~\ref{thm-Domar-3-quant} has the form 
    \[
    a^{\rho^-(\dist(x,B)/(3D))}
       =a^{W(\dist(x,B)) \dist(x,B)^{-\frac \al{1-\al}}}, 
    \]	
    where $\lim_{t\to 0^+} W(t)$ is finite and positive. For any constant $v>0$, 
    this majorant grows much faster near the set $B$ than the function $g(v\dist(\cdot, B))$. If $\al$ is close to $1$ (that is, the growth of $g$ is close to the limit one), then the exponent $-\frac \al{1-\al}$ in the above expression is close to $-\infty$.      
\end{example}
We end the article with the following open questions. 
\begin{question1}
Can the conclusion~\eqref{estim-cg} of Theorem ~\ref{thm-Domar-s^-b}
be substituted by an estimate 
\[
u(x)\le g(v\dist(x,B)), \quad x\in\Om\sm B, 
\]
where $v\in(0,1)$ is a constant? 
\end{question1}

By Corollary~\ref{cor:union-Lip-surf}, it is true for finite unions of closed subsets of Lipschitz curves. 

\begin{question2}
	Are the estimates in quantitative Domar's Theorems~\ref{thm-Domar-2-quant} and \ref{thm-Domar-3-quant} optimal in some sense? 
\end{question2}

\begin{question3}
	Is the estimate in Theorem~\ref{thm-Domar-A-B} optimal? 
\end{question3}

\begin{question4}
	In particular, does there exist a regular function $g(t)$, whose singularity at $0^+$ is stronger than any power $t^{-b}$, such that, 
	for a class of sets $A$,  
	\eqref{eq:estim-u1} implies an estimate $u(x)\le g(v \dist(x,B))$, where $v\in (0,1)$ is a constant? 
\end{question4}

\begin{question5}
	Suppose that $k=2$, and identify $\R^2$ with the complex plane. Are the estimates for subharmonic functions of the form $u(z)=\log |f(z)|$, where $f$ is analytic on $\Om\sm B$, better than those for general subharmonic functions? 
\end{question5}

 \section{Appendix}
 \label{Sec:Appendix}

\begin{proof}[Proof of Lemma~\ref{lem:dist.delta}]
Since $u$ is bounded on compact subsets of $\Omega$, 
by Corollary~\ref{cor:Domar.L1} we have 
\[
\dist(x_t,\partial\Omega) 
\le 
D\sum_{n=0}^{\infty}(\ell_{t+n-\lambda}+\dotsb+\ell_{t+n})^{1/k}.
\]
Clearly 
\[
\sum_{n=0}^{\infty}(\ell_{t+n-\lambda}+\dotsb+\ell_{t+n})^{1/k}
\le
\sum_{n=0}^{\infty}\pa{\ell_{t+n-\lambda}^{1/k}+\dotsb+\ell_{t+n}^{1/k}}
\le
(\lambda+1)\sum_{n=0}^{\infty}\ell_{t+n-\lambda}^{1/k},
\]
and by H\"{o}lder's inequality 
\begin{align*}
	\sum_{n=0}^{\infty}\ell_{t+n-\lambda}^{1/k} 
	&=
	\sum_{n=0}^{\infty}
	\frac{1}{(t+n-\lambda)^{\frac{k-1+\epsilon}{k}}}
	\cdot 
	(t+n-\lambda)^{\frac{k-1+\epsilon}{k}}\ell_{t+n-\lambda}^{1/k}\\
	&\le 
	\co{\sum_{n=0}^{\infty}\frac{1}{(t+n-\lambda)^{\frac{k-1+\epsilon}{k-1}}}}^{\frac{k-1}{k}}
	\cdot
	\co{\sum_{n=0}^{\infty}(t+n-\lambda)^{k-1+\epsilon}\ell_{t+n-\lambda}}^{\frac1k}.
\end{align*}
Now, on one hand we have 
\begin{align*}
\sum_{n=0}^{\infty}\frac{1}{(t+n-\lambda)^{\frac{k-1+\epsilon}{k-1}}}
&< 
\int_{-1}^{\infty}\frac{ds}{(t+s-\lambda)^{\frac{k-1+\epsilon}{k-1}}}\\
&= 
\co{\frac{(t+s-\lambda)^{-\frac{\epsilon}{k-1}}}{-\frac{\epsilon}{k-1}}}_{s=-1}^{s=\infty}\\
&=
\frac{k-1}{\epsilon}\cdot\frac{1}{(t-1-\lambda)^{\frac{\epsilon}{k-1}}}.
\end{align*}
Notice that $\{x\in\Omega:u(x)>a^{t-\lambda}\}$ is an open set 
(since $u$ is upper semicontinuous) 
that contains $x_t$, so it has positive Lebesgue measure. Therefore  
\[
0<m(\{a^{t-\lambda}\le u\})=\sum_{n=0}^{\infty}\ell_{t+n-\lambda},
\]
and consequently $\ell_{t+n-\lambda}>0$ for some $n\ge0$. Hence 
\[
0<\sum_{n=0}^{\infty}(t+n-\lambda)^{k-1+\epsilon}\ell_{t+n-\lambda}.
\]
Therefore 
\[
\dist(x_t,\partial\Omega)
<
D\delta(t)
\cdot
\co{\sum_{n=0}^{\infty}(t+n-\lambda)^{k-1+\epsilon}\ell_{t+n-\lambda}}^{\frac1k}.
\]
On the other hand, 
\begin{align*}
	\sum_{n=0}^{\infty}(t+n-\lambda)^{k-1+\epsilon}\ell_{t+n-\lambda}
	&\le 
	\sum_{n=0}^{\infty}
	\int_{\big\{t+n-\lambda\le \log_a^+u <t+n-\lambda+1\big\}}
	[\log_a^+u(x)]^{k-1+\epsilon}\, dx\\
	&= 
	\int_{\{\log_a^+u\ge t-\lambda\}}[\log_a^+u(x)]^{k-1+\epsilon}\, dx\\
	&\le 
	\int_{\{\log_a^+F\ge t-\lambda\}}[\log_a^+F(x)]^{k-1+\epsilon}\, dx,
\end{align*}
where in the first inequality we have used that 
\[
\ell_\nu=\{\nu\le \log_a^+u<\nu+1\},
\]
and in the last inequality we have used that $u\le F$. 
Therefore 
\[
\dist(x_t,\partial\Omega) 
< 
D\delta(t)\cdot\pa{\int_{\{\log_a^+F\ge t-\lambda\}}[\log_a^+F(x)]^{k-1+\epsilon}\, dx}^{\frac1k}.
\]
Now set $H(x):=[\log_a^+F(x)]^{k-1+\epsilon}$ and let $h$ be its distribution function. 
Then 
\begin{align*}
	h(s)
	=
	m(\{[\log_a^+F]^{k-1+\epsilon}>s\})
	=
	m(\{F>\exp(s^\frac{1}{k-1+\epsilon})\})
	=
	f_1(s^\frac{1}{k-1+\epsilon}). 
\end{align*}
Using \eqref{eq:distrib.funct} we get 
\begin{align*}
	\int_{\{\log_a^+F\ge t-\lambda\}}[\log_a^+F(x)]^{k-1+\epsilon}\, dx
	&= 
	\int_{\{H\ge(t-\lambda)^{k-1+\epsilon}\}}H(x)\, dx\\
	&= 
	\int_{(t-\lambda)^{k-1+\epsilon}}^{\infty}h(s)\, ds+(t-\lambda)^{k-1+\epsilon}h\big((t-\lambda)^{k-1+\epsilon}\big)\\
	&=
	\int_{(t-\lambda)^{k-1+\epsilon}}^{\infty}f_1(s^\frac{1}{k-1+\epsilon})\, ds+(t-\lambda)^{k-1+\epsilon}f_1(t-\lambda)\\
	&=
	\psi(t)^k.
\end{align*}
Summing up, $\dist(x_t,\partial\Omega)<D\varphi(t)$, as we wanted to prove. 
\end{proof}

\bibliographystyle{siam}
\bibliography{biblio-new}
\end{document}